\documentclass[11pt]{amsart}
\usepackage{amssymb}
\usepackage{graphicx}
\usepackage[font=small]{caption}
\usepackage{cite}
\usepackage[left=2.5cm,right=2.5cm,top=3cm,bottom=3cm]{geometry}
\usepackage{color}
\usepackage{mathtools}
\usepackage{amscd}
\usepackage{latexsym}
\usepackage{epsfig}
\usepackage{graphicx}
\usepackage{psfrag}
\usepackage{caption}
\usepackage{rotating}
\usepackage{mathtools}
\usepackage{enumitem}
\usepackage{longtable}
\usepackage[normalem]{ulem}

\newtheorem{thmx}{Theorem}

\allowdisplaybreaks[3]

\makeatletter
\@namedef{subjclassname@1991}{2020 Mathematics Subject Classification}
\makeatother

\newtheorem{theorem}{Theorem}
\newtheorem*{theorem*}{Theorem}
\newtheorem{proposition}[theorem]{Proposition}
\newtheorem{corollary}[theorem]{Corollary}

\theoremstyle{definition}
\newtheorem{definition}[theorem]{Definition}

\newtheorem*{question*}{Question}

\newcommand{\sll}[1]{\mkern-4mu\mathbin{/\mkern-5mu/}_{\mkern-4mu{#1}}}

\newcommand{\C}{\mathbb{C}}

\newcommand{\g}{\mathfrak{g}}
\newcommand{\h}{\mathfrak{h}}

\renewcommand{\exp}{\mathrm{exp}}

\numberwithin{equation}{section}

\title[The double Gelfand--Cetlin system]{The double Gelfand--Cetlin system, invariance of polarization, and the Peter--Weyl theorem}

\author[Peter Crooks]{Peter Crooks}
\author[Jonathan Weitsman]{Jonathan Weitsman}
\address[Peter Crooks]{Department of Mathematics and Statistics\\ Utah State University \\ 3900 Old Main Hill \\ Logan, UT 84322, USA}
\email{peter.crooks@usu.edu}
\address[Jonathan Weitsman]{Department of Mathematics \\ Northeastern University \\ 360 Huntington Avenue \\ Boston, MA 02115, USA}
\email{j.weitsman@northeastern.edu}

\begin{document}

\begin{abstract} 
The bundle map $T^*\hspace{-2pt}\operatorname{U}(n)\longrightarrow\operatorname{U}(n)$ provides a real polarization of the cotangent bundle $T^*\hspace{-2pt}\operatorname{U}(n)$, and yields the geometric quantization $Q_1(T^*\hspace{-2pt}\operatorname{U}(n)) = L^2(\operatorname{U}(n))$. We use the Gelfand--Cetlin systems of Guillemin and Sternberg to show that $T^*\hspace{-2pt}\operatorname{U}(n)$ has a different real polarization with geometric quantization $Q_2(T^*\hspace{-2pt}\operatorname{U}(n))=  \bigoplus_\alpha V_\alpha \otimes V_\alpha^*$, where the sum is over all dominant integral weights $\alpha$ of $\operatorname{U}(n)$.  The Peter--Weyl theorem, which states that these two quantizations are isomorphic, may therefore be interpreted as an instance of ``invariance of polarization" in geometric quantization.
\end{abstract}

\maketitle
\begin{scriptsize}
\setcounter{secnumdepth}{2}
\setcounter{tocdepth}{2}
\tableofcontents
\end{scriptsize}

\section{Introduction}
\subsection{Geometric quantization, real polarizations, and invariance of polarization} Let $(M,\omega)$ be an integral symplectic manifold. Suppose that $\mathcal{L}\longrightarrow M$ is a Hermitian line bundle with connection $\nabla$ of curvature $\omega$. The geometric quantization program of Kirillov--Kostant--Souriau (KKS) \cite{KostantSymplectic,Souriau,Kirillov} would assign to the \textit{prequantum system} $(M,\omega,\mathcal{L},\nabla)$ a complex vector space $Q(M)=Q(M,\omega,\mathcal{L},\nabla)$, called the \textit{geometric quantization} of $M$.\footnote{More generally, the quantization $Q(M)$ may be equipped with a Hermitian metric, or with further structure.  In this paper, we focus on the construction of $Q(M)$ as a vector space.}  One of the most basic test cases for this program involves taking $M=T^*N$ for a smooth manifold $N$, and $\omega$ to be the exterior derivative of the tautological one-form on $T^*N$. This one-form provides a connection on $\mathcal{L}=T^*N\times\mathbb{C}$ of curvature $\omega$. Considerations in quantum mechanics suggest that any physically reasonable approach to geometric quantization should satisfy $Q(T^*N)=L^2(N)$, the vector space of square-integrable functions $N\longrightarrow\mathbb{C}$.

To generalize this to symplectic manifolds $M$ other than cotangent bundles, one seeks auxillary data analogous to the cotangent bundle projection $T^*N\longrightarrow N$. One choice of such data is a {\em real polarization,} which is morally a foliation of $M$ by Lagrangian submanifolds.  The simplest case of a real polarization is a foliation arising from a Lagrangian fibration $\pi:M\longrightarrow B$. The generalization of the quantization $Q(T^*N)=L^2(N)$ is obtained by defining $Q(M)$ as the complex vector space of sections $s:M\longrightarrow\mathcal{L}$ that are covariantly constant along the fibers of $\pi$.  

The results of Sniatycki \cite{Sniatycki} show that requiring the foliation to be nonsingular and fibrating is highly restrictive;  if $M$ is compact, the only examples are abelian varieties.  
In order to move beyond abelian varieties, many examples in the literature on geometric quantization were considered where the map $\pi$ was only generically a Lagrangian fibration. The real polarization appearing in this paper will also be of this type.  A further issue arises with quantization: if the fibers of $\pi$ are compact, there are no fiberwise-constant sections; such sections instead exist only on a discrete set of fibers $\pi^{-1}(b)$, corresponding to points $b\in B_{\text{bs}} \subset B,$ where $B_{\text{bs}}\subset B$ is a discrete subset of $B,$ the set of \textit{Bohr--Sommerfeld points}. The work of Sniatycki \cite{Sniatycki} nevertheless motivates the following construction, which has been used in many contexts: the geometric quantization of a prequantum system $(M,\omega,\mathcal{L},\nabla)$ in a real polarization given by a map $\pi:M\longrightarrow B$ is defined as $$Q(M)\coloneqq\bigoplus_{b\in B_{\text{bs}}}\mathbb{C}\langle s_b\rangle,$$ where $s_b$ is any non-vanishing, covariantly constant section of the pullback of $(\mathcal{L},\nabla)$ to $\pi^{-1}(b)$ for $b\in B_{\text{bs}}$. This definition is made with the understanding that various ad hoc modifications may need to be made to account for singular fibers of $\pi$.  We will also follow the approach of the literature in this regard.

Despite the ambiguity in the constructions, this approach to quantization in a real polarization yields interesting results in many examples. In particular, it turns out to agree with K\"ahler quantization and/or quantization in a different real polarization in examples as diverse as toric manifolds \cite{Hamilton,Hamilton2}, abelian varieties \cite{Mumford,Sniatycki2}, flag varieties \cite{GuilleminSternbergGC}, and certain moduli spaces \cite{JeffreyWeitsman}. These are examples of the recurring phenomenon of \textit{invariance of polarization}.   The purpose of this paper is to present another example of this phenomenon, in the case where $M=T^*\hspace{-2pt}\operatorname{U}(n)$.

It is illustrative to consider the case of $M=T^*\hspace{-2pt}\operatorname{U}(1)=S^1\times\mathbb{R}.$ In explicit coordinates $p$ on $\mathbb{R}$ and $\vartheta$ on $S^1$, the tautological one-form and symplectic form on $T^*\hspace{-2pt}\operatorname{U}(1)$ are given by $\theta=p\mathrm{d}\vartheta$ and $\omega=\mathrm{d}\theta=\mathrm{d}p\wedge\mathrm{d}\vartheta$, respectively. Equipping $T^*\hspace{-2pt}\operatorname{U}(1)$ with the trivial Hermitian complex line bundle $\mathcal{L} = T^*\hspace{-2pt}\operatorname{U}(1) \times \C,$ with compatible connection given by the one-form $\vartheta,$ then gives a prequantum system. The natural projections $\pi_1:T^*\hspace{-2pt}\operatorname{U}(1)\longrightarrow S^1$ and $\pi_2:T^*\hspace{-2pt}\operatorname{U}(1)\longrightarrow\mathbb{R}$ are both real polarizations. We have seen that the quantization given by $\pi_1$ is $Q_1(T^*\hspace{-2pt}\operatorname{U}(1))=L^2(\operatorname{U}(1))$. In the real polarization given by $\pi_2$, covariantly constant sections exist only on the fibers of $\pi_2$ lying over integer points $k\in\mathbb{Z}\subset\mathbb{R}$; a non-vanishing, covariantly constant section on $\pi^{-1}(k)$ is $s_k:S^1\longrightarrow\mathbb{C}$, $\theta\mapsto\mathrm{e}^{ik\theta}$, for $k\in\mathbb{Z}$. The geometric quantization in the real polarization given by $\pi_2$ is therefore given by $$Q_2(T^*\hspace{-2pt}\operatorname{U}(1))=\bigoplus_{k\in\mathbb{Z}}\mathbb{C}\langle \mathrm{e}^{ik\theta}\rangle.$$
\begin{center}
\includegraphics[width=10cm,height=5cm,angle=0]{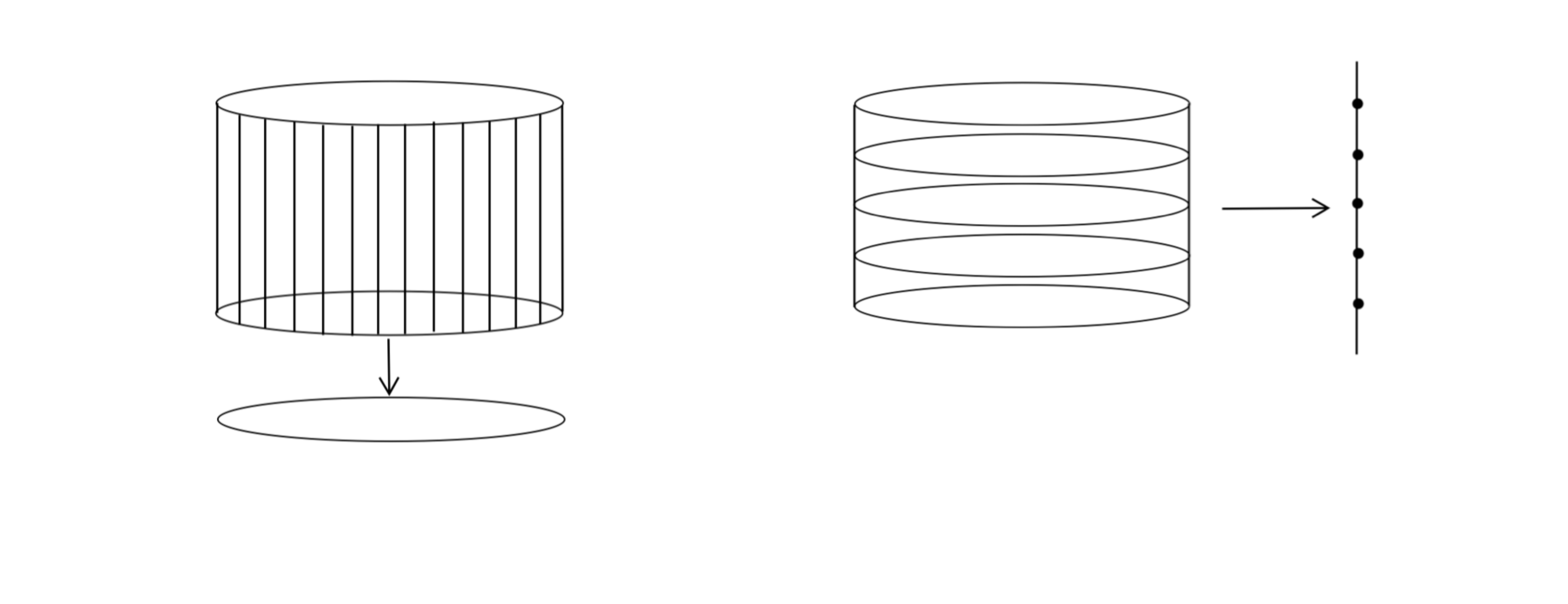}
\vspace{-25pt}
\captionof{figure}{Two real polarizations of $T^*\hspace{-2pt}\operatorname{U}(1)$} 
\end{center}
\vspace{10pt}
\noindent Fourier analysis then implies that $Q_1(T^*\hspace{-2pt}\operatorname{U}(1))$ is the completion of $Q_2(T^*\hspace{-2pt}\operatorname{U}(1))$ (in the $L^2$ metric on $Q_1(T^*\hspace{-2pt}\operatorname{U}(1))= L^2(\operatorname{U}(1))$), and so gives an example of invariance of polarization.

\subsection{Real polarizations of $T^*\hspace{-2pt}\operatorname{U}(n)$} The purpose of this paper is to show that a similar phenomenon occurs if $\operatorname{U}(1)$ is replaced with $\operatorname{U}(n)$. A first step is to recall that the bundle projection $T^*\hspace{-2pt}\operatorname{U}(n)\longrightarrow\operatorname{U}(n)$ is a real polarization with quantization $Q_1(T^*\hspace{-2pt}\operatorname{U}(n))=L^2(\operatorname{U}(n))$. To obtain a second real polarization of $T^*\hspace{-2pt}\operatorname{U}(n)$ and quantization $Q_2(T^*\hspace{-2pt}\operatorname{U}(n))$, we set $\text{b}\coloneqq\frac{n(n+1)}{2}$ and use the Gelfand--Cetlin system $$\lambda:\mathfrak{u}(n)^*\longrightarrow\mathbb{R}^{\text{b}}$$ of Guillemin--Sternberg \cite{GuilleminSternbergGC}: one has $$\lambda(\xi)=(\lambda_{01}(\xi),\ldots,\lambda_{0n}(\xi),\lambda_{11}(\xi),\ldots,\lambda_{1(n-1)}(\xi),\ldots,\lambda_{(n-1)1}(\xi))$$ for all skew-Hermitian $n\times n$ matrices $\xi$, where $\lambda_{j1}(\xi)\geq\cdots\geq\lambda_{j(n-j)}(\xi)$ are the eigenvalues of the bottom-right $(n-j)\times(n-j)$ corner of $\xi$ for all $j\in\{0,\ldots,n-1\}$. Another ingredient is the moment map
$$\phi:T^*\hspace{-2pt}\operatorname{U}(n)\longrightarrow\mathfrak{u}(n)^*\times\mathfrak{u}(n)^*$$ for the Hamiltonian action of $\operatorname{U}(n)\times\operatorname{U}(n)$ on $T^*\hspace{-2pt}\operatorname{U}(n)$. The composite map
$$\varphi\coloneqq(\lambda,\lambda)\circ\phi:T^*\hspace{-2pt}\operatorname{U}(n)\longrightarrow\mathbb{R}^{\text{b}}\times\mathbb{R}^{\text{b}}=\mathbb{R}^{2\text{b}}$$ then restricts to a moment map for a Hamiltonian action of the torus $\mathbb{T}\times\mathbb{T}$ on an open dense subset $(T^*\hspace{-2pt}\operatorname{U}(n))_{\text{s-reg}}\subset T^*\hspace{-2pt}\operatorname{U}(n)$, where $\mathbb{T}\coloneqq\operatorname{U}(1)^{\text{b}}$. Taking a quotient of $\mathbb{T}\times\mathbb{T}$ by the kernel of this action yields a torus $\mathcal{T}$. The action of $\mathcal{T}$ on  $(T^*\hspace{-2pt}\operatorname{U}(n))_{\text{s-reg}}$ turns out to be free, and the map $\varphi$ takes values in $$\mathcal{B}\coloneqq\mathrm{Lie}(\mathcal{T})^*\subset \mathbb{R}^{2\text{b}}.$$ One finds the $\mathcal{T}$-moment map
\begin{equation}\label{Equation: Real polarization}\varphi\big\vert_{(T^*\hspace{-2pt}\operatorname{U}(n))_{\text{s-reg}}}:(T^*\hspace{-2pt}\operatorname{U}(n))_{\text{s-reg}}\longrightarrow\mathcal{B}\end{equation} to be a real polarization whose non-empty fibers are precisely the orbits of $\mathcal{T}$ in $T^*\hspace{-2pt}\operatorname{U}(n))_{\text{s-reg}}$.

We next describe the set of Bohr--Sommerfeld points (a.k.a. Bohr--Sommerfeld set) of the real polarization constructed above. To this end, suppose that $\alpha=(\alpha_1,\ldots,\alpha_{n})$ is a non-increasing sequence of real numbers. Write $\mathcal{O}_{\alpha}\subset\mathfrak{u}(n)^*$ for the coadjoint orbit of $\operatorname{U}(n)$ through $\mathrm{diag}(\alpha_{1},\ldots,\alpha_n)\in\mathfrak{u}(n)^*$. The compact subset $\mathrm{GC}_\alpha\subset\mathbb{R}^{\frac{n(n-1)}{2}}$ satisfying $\lambda(\mathcal{O}_{\alpha})=\{\alpha\}\times\mathrm{GC}_{\alpha}$ is the \textit{Gelfand--Cetlin polytope} of $\mathcal{O}_{\alpha}$, and features in the following result.    

\begin{thmx}\label{Theorem: First}
The Bohr--Sommerfeld set of \eqref{Equation: Real polarization} is $\varphi((T^*\hspace{-2pt}\operatorname{U}(n))_{\emph{s-reg}})\cap\mathbb{Z}^{2\emph{b}}$. Furthermore, a point $(\alpha_1,\ldots,\alpha_{2\emph{b}})\in\mathbb{R}^{2\emph{b}}$ belongs to this set if and only if it satisfies the following conditions:
\begin{itemize}
\item[\textup{(i)}] $(\alpha_1,\ldots,\alpha_{2\emph{b}})\in\mathbb{Z}^{2\emph{b}}$;
\item[\textup{(ii)}] $\alpha_1>\cdots>\alpha_{n}$;
\item[\textup{(iii)}] $\alpha_{\emph{b}+1+j}=-\alpha_{n-j}$ for all $j\in\{0,\ldots,n-1\}$;
\item[\textup{(iv)}] $(\alpha_{n+1},\ldots,\alpha_{\emph{b}})\in\mathrm{interior}(\mathrm{GC}_{\alpha})$, where $\alpha\coloneqq(\alpha_1,\ldots,\alpha_{n})$;
\item[\textup{(v)}] $(\alpha_{\emph{b}+n+1},\ldots,\alpha_{\emph{2b}})\in\mathrm{interior}(\mathrm{GC}_{\alpha^*})$, where $\alpha^*\coloneqq(-\alpha_{n},\ldots,-\alpha_1)$.
\end{itemize}
\end{thmx}

\subsection{The quantization}\label{Subsection: The quantization}
Suppose that $\alpha=(\alpha_1,\ldots,\alpha_{n})\in\mathbb{R}^n$ is non-increasing. Guillemin and Sternberg \cite{GuilleminSternbergGC} consider a continuous map $\lambda_{\alpha}:\mathcal{O}_{\alpha}\longrightarrow\mathbb{R}^{\frac{1}{2}\dim\mathcal{O}_{\alpha}}$ that restricts to a real polarization on an open dense subset $(\mathcal{O}_{\alpha})_{\text{s-reg}}\subset\mathcal{O}_{\alpha}$. They also note that $\mathcal{O}_{\alpha}$ carries a prequantum system in the special case of $\alpha\in\mathbb{Z}^n$. This real polarization then gives rise to a well-defined Bohr--Sommerfeld set for $\lambda_{\alpha}\big\vert_{(\mathcal{O}_{\alpha})_{\text{s-reg}}}$; it consists of the integral points in $\lambda_{\alpha}((\mathcal{O}_{\alpha})_{\text{s-reg}})$. As explained in Sections 5 and 6 of \cite{GuilleminSternbergGC}, it is natural to redefine this Bohr--Sommerfeld set as the collection of integral points in $\lambda_{\alpha}(\mathcal{O}_{\alpha})=\mathrm{GC}_{\alpha}$, and thereby think of oneself as quantizing all of $\mathcal{O}_{\alpha}$. One appealing consequence of this modified definition is that $\dim Q(\mathcal{O}_{\alpha})=\dim V_{\alpha}$, where $V_{\alpha}$ is the irreducible complex $\operatorname{U}(n)$-module of highest weight $\alpha$ \cite[Theorem 6.1]{GuilleminSternbergGC}.      

We adopt an analogous approach for purposes of quantizing $T^*\hspace{-2pt}\operatorname{U}(n)$. This amounts to replacing the Bohr--Sommerfeld set $\varphi((T^*\hspace{-2pt}\operatorname{U}(n))_{\text{s-reg}})\cap\mathbb{Z}^{2\text{b}}$ in Theorem \ref{Theorem: First} with $\varphi(T^*\hspace{-2pt}\operatorname{U}(n))\cap\mathbb{Z}^{2\text{b}}$. A point $(\alpha_1,\ldots,\alpha_{2\text{b}})\in\mathbb{R}^{2\text{b}}$ then belongs to this new Bohr--Sommerfeld set if and only if it satisfies the following conditions: 
\begin{itemize}
\item[\textup{(i)}] $(\alpha_1,\ldots,\alpha_{2\text{b}})\in\mathbb{Z}^{2\text{b}}$;
\item[\textup{(ii)}] $\alpha_1\geq\cdots\geq\alpha_{n}$;
\item[\textup{(iii)}] $\alpha_{\text{b}+1+j}=-\alpha_{n-j}$ for all $j\in\{0,\ldots,n-1\}$;
\item[\textup{(iv)}] $(\alpha_{n+1},\ldots,\alpha_{\text{b}})\in\mathrm{GC}_{\alpha}$, where $\alpha\coloneqq(\alpha_1,\ldots,\alpha_{n})$;
\item[\textup{(v)}] $(\alpha_{\text{b}+n+1},\ldots,\alpha_{\text{2b}})\in\mathrm{GC}_{\alpha^*}$, where $\alpha^*\coloneqq(-\alpha_{n},\ldots,-\alpha_1)$.
\end{itemize} In this way, a Bohr--Sommerfeld point is the data of a dominant integral weight $\alpha$ of $\operatorname{U}(n)$, an integral point in $\mathrm{GC}_{\alpha}$, and an integral point in $\mathrm{GC}_{\alpha^*}$. We also recognize that $\alpha^*$ is the highest weight of $V_{\alpha}^*$ for each dominant integral weight $\alpha$. These considerations justify the following result.

\begin{thmx} The geometric quantization of $T^*\hspace{-2pt}\operatorname{U}(n)$ associated with the real polarization \eqref{Equation: Real polarization} is
$$Q_2(T^*\hspace{-2pt}\operatorname{U}(n))=\bigoplus_{\alpha}V_{\alpha}\otimes V_{\alpha}^*,$$ where the sum is taken over all dominant integral weights $\alpha$ of $\operatorname{U}(n)$.
\end{thmx}
  
The Peter--Weyl theorem allows one to identify (the completion of) $Q_2(T^*\hspace{-2pt}\operatorname{U}(n))$ with $Q_1(T^*\hspace{-2pt}\operatorname{U}(n))=L^2(\operatorname{U}(n))$, and so provides an instance of invariance of polarization.

\subsection{Organization} We devote Section \ref{Section: Double GC} to the definition, underlying theory, and basic properties of the double Gelfand--Cetlin system on $T^*\hspace{-2pt}\operatorname{U}(n)$. This section begins with Section \ref{Subsection: GC theory}, a brief overview of the pertinent parts of Gelfand--Cetlin theory in arbitrary Lie type. A subsequent specialization to the case $G=\operatorname{U}(n)$ is given in Section \ref{Subsection: Special case}. In Section \ref{Subsection: Hamiltonian geometry}, we briefly review the Hamiltonian geometry of the cotangent bundle of a compact connected Lie group. The double Gelfand--Cetlin system on $T^*\hspace{-2pt}\operatorname{U}(n)$ is then defined and studied in Section \ref{Subsection: Double GC}. In Section \ref{Subsection: Images}, we compute the image of the moment map underlying this system.

Section \ref{Section: Quantization} is entirely concerned with quantizing $T^*\hspace{-2pt}\operatorname{U}(n)$ in the real polarization coming from the double Gelfand--Cetlin system. In Section \ref{Subsection: BS}, we compute the Bohr--Sommerfeld set arising from this polarization. This leads to Section \ref{Subsection: Peter--Weyl}, in which we use the Bohr--Sommerfeld set and Peter--Weyl theorem to present an instance of invariance of polarization.

\subsection*{Acknowledgements} P.C. was supported by a startup grant at Utah State University. J.W. was supported by Simons Collaboration Grant \# 579801.

\section{The double Gelfand--Cetlin system}\label{Section: Double GC}
\subsection{Some Gelfand--Cetlin theory}\label{Subsection: GC theory} Let $G$ be a compact connected Lie group with rank $\ell$, Lie algebra $\g$, exponential map $\exp:\g\longrightarrow G$, and fixed $G$-invariant inner product $\g\otimes_{\mathbb{R}}\g\longrightarrow\mathbb{R}$. We will have occasion to consider the quantity $$\text{b}\coloneqq\frac{1}{2}(\dim\g+\ell).$$ Let us also suppose that $G=G_0\supset G_1\supset\cdots\supset G_m$ is a descending filtration of $G$ by connected closed subgroups with respective Lie algebras $\g=\g_0\supset\g_1\supset\cdots\supset\g_m$ and ranks $\ell=\ell_0\geq\ell_1\geq\cdots\geq\ell_m$, and that the condition $\text{b}=\sum_{j=0}^m\ell_j$ is satisfied. Choose maximal tori $T_j\subset G_j$ for $j\in\{0,\ldots,m\}$. Consider the Lie algebra $\mathfrak{t}_j\subset\g_j$ for each $j\in\{0,\ldots,m\}$, and choose a fundamental Weyl chamber $(\mathfrak{t}_j)_{+}\subset\mathfrak{t}_j$. Observe that our $G$-invariant inner product induces an isomorphism $\mathfrak{g}_j\cong\mathfrak{g}_j^*$ of $G_j$-representations. The subsets $\mathfrak{t}_j$ and $(\mathfrak{t}_j)_+$ of $\g_j$ are thereby identified with subsets $\mathfrak{t}_j^*$ and $(\mathfrak{t}_j^*)_{+}$ of $\mathfrak{g}_j^*$, respectively. These considerations give rise to a \text{sweeping map} $\pi_j:\mathfrak{g}_j^*\longrightarrow(\mathfrak{t}_j^*)_{+}$; it is defined by the property that $(G_j\cdot\xi)\cap(\mathfrak{t}_j^*)_{+}=\{\pi_j(\xi)\}$ for all $\xi\in\mathfrak{g}_j^*$, where $G_j\cdot\xi\subset\g_j^*$ denotes the coadjoint orbit of $G_j$ through $\xi$. 

Fix an index $j\in\{0,\ldots,m\}$. Given any maximal torus $H\subset G_j$ with Lie algebra $\h\subset\g_j$, consider the lattice $$\Lambda^{\vee}(H)\coloneqq\frac{1}{2\pi}\mathrm{ker}\left(\exp\big\vert_{\h}:\h\longrightarrow H\right)\subset\h.$$ The weight lattice of $H$ is then given by
$$\Lambda(H)\coloneqq\{\alpha\in\h^*:\alpha(x)\in\mathbb{Z}\text{ for all }x\in\Lambda^{\vee}(H)\}\subset\h^*.$$
On the other hand, let us choose a $\mathbb{Z}$-basis $\{\phi_{j1},\ldots,\phi_{j\ell_j}\}$ of the lattice $\Lambda^{\vee}(T_j)$. The functions 
$$\nu_{jk}\coloneqq\phi_{jk}\circ\pi_j:\g_j^*\longrightarrow\mathbb{R},\quad k\in\{1,\ldots,\ell_j\}$$ are then observed to be $G_j$-invariant and continuous. These functions are smooth on $$(\mathfrak{g}_j^*)_{\text{reg}}\coloneqq\{\xi\in\g_j^*:\dim(\g_{j})_{\xi}=\ell_j\},$$ where $(\g_j)_{\xi}\subset\g_j$ denotes the $\g_j$-centralizer of $\xi\in\g_j^*$. Let us also write $(G_j)_{\xi}\subset G_j$ for the $G_j$-stabilizer of $\xi\in\g_j^*$ under the coadjoint action. The next result then plays an important role in what follows. 

\begin{proposition}\label{Proposition: Integral}
Suppose that $j\in\{0,\ldots,m\}$ and $\xi\in(\g_j^*)_{\emph{reg}}$. The restriction $\xi\big\vert_{(\g_j)_{\xi}}\in(\g_j)_{\xi}^*$ is in the weight lattice $\Lambda((G_j)_{\xi})$ if and only if $\nu_{jk}(\xi)\in\mathbb{Z}$ for all $k\in\{1,\ldots,\ell_j\}$.
\end{proposition}

\begin{proof}
An application of \cite[Proposition 5]{CrooksWeitsman} reveals that $\{\mathrm{d}_{\xi}\nu_{j1},\ldots,\mathrm{d}_{\xi}\nu_{j\ell_j}\}$ is a $\mathbb{Z}$-basis of $\Lambda^{\vee}((G_j)_{\xi})\subset(\g_j)_{\xi}$. It follows that $\xi\big\vert_{(\g_j)_{\xi}}$ is in $\Lambda((G_j)_{\xi})$ if and only if $\xi(\mathrm{d}_{\xi}\nu_{jk})\in\mathbb{Z}$ for all $k\in\{1,\ldots,\ell_j\}$. On the other hand, it is clear that $\xi(\mathrm{d}_{\xi}\nu_{jk})=\nu_{jk}(\xi)$ for all $k\in\{1,\ldots,\ell_j\}$.
\end{proof}

This gives context for the following slightly rephrased version of \cite[Proposition 7]{CrooksWeitsman}, in which we write
$\mathbb{T}_j\coloneqq\operatorname{U}(1)^{\ell_j}$ for $j\in\{0,\ldots,m\}$.

\begin{proposition}\label{Proposition: Integration}
If $j\in\{0,\ldots,m\}$ and $\xi\in(\g_j^*)_{\emph{reg}}$, then the linear map $$\mathbb{R}^{\ell_j}\longrightarrow(\g_j)_{\xi},\quad (x_1,\ldots,x_{\ell_j})\mapsto\sum_{k=1}^{\ell_j}x_k\mathrm{d}_{\xi}\nu_{jk}$$ integrates to a Lie group isomorphism $(\tau_j)_{\xi}:\mathbb{T}_j\overset{\cong}\longrightarrow (G_j)_{\xi}$.
\end{proposition}

Let $\sigma_j:\g^*\longrightarrow\g_j^*$ denote the restriction map for each $j\in\{0,\ldots,m\}$. This allows us to form the functions \begin{equation}\label{Equation: Eigenvalue functions}\lambda_{jk}\coloneqq\nu_{jk}\circ\sigma_j:\g^*\longrightarrow\mathbb{R}\end{equation} for $j\in\{0,\ldots,m\}$ and $k\in\{1,\ldots,\ell_j\}$, and subsequently set
$$\label{Equation: Enumeration}\lambda\coloneqq(\lambda_1,\ldots,\lambda_{\text{b}})\coloneqq(\lambda_{01},\ldots,\lambda_{0\ell},\lambda_{11},\ldots,\lambda_{1\ell_1}\ldots,\lambda_{m1},\ldots,\lambda_{m\ell_m}):\g^*\longrightarrow\mathbb{R}^{\text{b}}.$$ It follows that $\lambda$ is smooth on the open subset $$\bigcap_{j=0}^m\sigma_j^{-1}((\g_j^*)_{\text{reg}})\subset\g^*,$$ in which context we define the \textit{strongly regular} locus $$\g^*_{\text{s-reg}}\coloneqq\left\{\xi\in\bigcap_{j=0}^m\sigma_j^{-1}((\g_j^*)_{\text{reg}}):\mathrm{rank}\hspace{2pt}\mathrm{d}_{\xi}\lambda=\text{b}\right\}.$$ On the other hand, consider the product torus
$$\mathbb{T}\coloneqq\prod_{j=0}^m\mathbb{T}_j$$ and its Lie algebra $\mathbb{R}^{\text{b}}=\prod_{j=0}^m\mathbb{R}^{\ell_j}$.
The following are direct consequences of \cite[Theorem 3.4]{GuilleminSternbergGC} and \cite[Proposition 8]{CrooksWeitsman}.

\begin{proposition}\label{Proposition: Summary}
Let $M$ be a Hamiltonian $G$-space with moment map $\mu:M\longrightarrow\g^*$.
\begin{itemize}
\item[\textup{(i)}] The composite map $\lambda\circ\mu:M\longrightarrow\mathbb{R}^{\text{b}}$ restricts to a moment map for a Hamiltonian action of $\mathbb{T}$ on $\mu^{-1}(\g^*_{\emph{s-reg}})$.
\item[\textup{(ii)}] Suppose that $\xi\in\g^*_{\emph{s-reg}}$ and $m\in\mu^{-1}(\xi)$. One then has \begin{equation}\label{Equation: Action} t\cdot m=(\tau_j)_{\sigma_j(\xi)}(t)\cdot m\end{equation} for all $j\in\{0,\ldots,m\}$ and $t\in\mathbb{T}_j$, where $(\tau_j)_{\sigma_j(\xi)}:\mathbb{T}_j\overset{\cong}\longrightarrow (G_j)_{\sigma_j(\xi)}$ is the isomorphism in Proposition \ref{Proposition: Integration}, and the left and right-hand sides of \eqref{Equation: Action} denote the actions of $t\in\mathbb{T}_j\subset\mathbb{T}$ on $\mu^{-1}(\g^*_{\emph{s-reg}})$ and $(\tau_j)_{\sigma_j(\xi)}(t)\in(G_j)_{\sigma_j(\xi)}\subset G_j$ on $M$, respectively.
\end{itemize}
\end{proposition}

\subsection{The case $G=\operatorname{U}(n)$}\label{Subsection: Special case}
It will be important to specialize some of the above-mentioned constructions to the case in which $G=\operatorname{U}(n)$ is the group of unitary $n\times n$ matrices. In this case, $\ell=n$, $\text{b}=\frac{n(n+1)}{2}$, and $\g=\operatorname{u}(n)$ is the real Lie algebra of skew-Hermitian $n\times n$ matrices. We will take $$\g\otimes_{\mathbb{R}}\g\longrightarrow\mathbb{R},\quad x\otimes y\mapsto -\mathrm{tr}(xy)$$ as our $G$-invariant inner product. By composing the induced isomorphism $\g^*\overset{\cong}\longrightarrow\g$ with multiplication by $i$, we will regard $\g^*$ as the real vector space of Hermitian $n\times n$ matrices. The pairing between $\g$ and $\g^*$ is then given by $$\xi(x)=-i\mathrm{tr}(x\xi)$$ for all $x\in\g$ and $\xi\in\g^*$. 

Let us consider the subgroup $$G_j\coloneqq\left\{\left[\begin{array}{ c | c }
    I_{j} & 0 \\
    \hline
    0 & A
 \end{array}\right]:A\in\operatorname{U}(n-j)\right\}\subset G$$
and its Lie algebra $\g_j\subset\g$ for $j\in\{0,\ldots,n-1\}$. It is clear that $G=G_0\supset G_1\supset\cdots\supset G_{n-1}$ and $\g=\g_0\supset\g_1\supset\cdots\supset\g_{n-1}$. We will identify $\g_j$ with the Lie algebra $\mathfrak{u}(n-j)$ of skew-Hermitian $n\times n$ matrices, and use the approach outlined above to identify $\g_j^*$ with the vector space of Hermitian $(n-j)\times (n-j)$ matrices. The restriction map $\sigma_j:\g^*\longrightarrow\g_j^*$ then sends $\xi\in\g^*$ to the bottom-right $(n-j)\times(n-j)$ corner of $\xi$.

Observe that $$T_j\coloneqq\left\{\left[\begin{array}{ c | ccc }
    I_j & & 0 & \\
    \hline
     & z_1 & &\\
    0 & & \ddots & \\
     & & & z_{n-j}
  \end{array}\right]:z_1,\ldots,z_{n-j}\in\operatorname{U}(1)\right\}\subset G_j$$ is a maximal torus with Lie algebra
$$\mathfrak{t}_j\coloneqq\left\{\left[\begin{array}{ c | ccc }
    0 & & 0 & \\
    \hline
     & ia_1 & &\\
    0 & & \ddots & \\
     & & & ia_{n-j}
  \end{array}\right]:a_1,\ldots,a_{n-j}\in\mathbb{R}\right\}\subset\g_j$$
for each $j\in\{0,\ldots,n-1\}$. For each $k\in\{1,\ldots,n-j\}$, let $\phi_{jk}\in\mathfrak{t}_j$ be what results from taking $a_k=1$ and $a_r=0$ for all $r\neq k$. It follows that $\{\phi_{j1},\ldots,\phi_{j(n-j)}\}$ is a $\mathbb{Z}$-basis of the lattice $\Lambda^{\vee}(T_j)\subset\mathfrak{t}_j$. We also have the fundamental Weyl chamber
$$(\mathfrak{t}_j)_{+}\coloneqq\left\{\left[\begin{array}{ c | ccc }
    0 & & 0 & \\
    \hline
     & ia_1 & &\\
    0 & & \ddots & \\
     & & & ia_{n-j}
  \end{array}\right]:\let\scriptstyle\textstyle\substack{a_1,\ldots,a_{n-j}\in\mathbb{R}\\ a_1\geq\cdots\geq a_{n-j}}\right\}\subset\mathfrak{t}_j$$ and corresponding subset
$$(\mathfrak{t}^*_j)_{+}\coloneqq\left\{\left[\begin{array}{ c | ccc }
    0 & & 0 & \\
    \hline
     & a_1 & &\\
    0 & & \ddots & \\
     & & & a_{n-j}
  \end{array}\right]:\let\scriptstyle\textstyle\substack{a_1,\ldots,a_{n-j}\in\mathbb{R}\\ a_1\geq\cdots\geq a_{n-j}}\right\}\subset\g_j^*\subset\g^*.$$ The functions $\lambda_{jk}:\g^*\longrightarrow\mathbb{R}$ in \eqref{Equation: Eigenvalue functions} are then defined on each $\xi\in\g^*$ as follows: $$\lambda_{j1}(\xi)\geq\cdots\geq\lambda_{j(n-j)}(\xi)$$ are the eigenvalues of the bottom-right $(n-j)\times(n-j)$ corner of $\xi$. One also has \begin{equation}\label{Equation: SR}\g^*_{\text{s-reg}}=\left\{\xi\in\g^*:\let\scriptstyle\textstyle\substack{\lambda_{j1}(\xi)>\cdots>\lambda_{j(n-j)}(\xi)\text{ for all }j\in\{0,\ldots,n-1\}\\ \lambda_{0k}(\xi)>\cdots>\lambda_{(n-k)k}(\xi)\text{ for all }k\in\{1,\ldots,n\}}\right\}\end{equation} as the strongly regular locus.

\subsection{The Hamiltonian geometry of $T^*G$}\label{Subsection: Hamiltonian geometry} Retain the notation and conventions used in Section \ref{Subsection: GC theory}. Consider the cotangent bundle $T^*G$, tautological one-form $\theta\in\Omega^1(T^*G)$, and symplectic form $\omega\coloneqq\mathrm{d}\theta\in\Omega^2(T^*G)$.

Now let $\mathrm{Ad}:G\longrightarrow\operatorname{GL}(\g^*)$ denote the coadjoint representation of $G$. Let us also consider the action of $G\times G$ on $G$ defined by $$(g_1,g_2)\cdot g\coloneqq g_1gg_2^{-1},\quad (g_1,g_2)\in G\times G,\text{ }g\in G,$$ as well as its cotangent lift to a Hamiltonian $(G\times G)$-action on $(T^*G,\omega)$.  Using the left trivialization to identify $T^*G$ with $G\times\g^*$, one finds this Hamiltonian action to be given by $$(g_1,g_2)\cdot(g,\xi)=(g_1gg_2^{-1},\mathrm{Ad}^*_{g_2}(\xi)),\quad (g_1,g_2)\in G\times G,\text{ }(g,\xi)\in G\times\g^*=T^*G.$$ It is also straightforward to verify that $$\phi\coloneqq(\phi_1,\phi_2):T^*G\longrightarrow\g^*\times\g^*,\quad (g,\xi)\mapsto (\mathrm{Ad}^*_g(\xi),-\xi),\quad (g,\xi)\in G\times\g^*=T^*G$$ is a moment map. 

Given $x=(x_1,x_2)\in\g\times\g$, consider the function
$$\phi^x:T^*G\longrightarrow\mathbb{R},\quad\beta\mapsto(\phi_1(\beta))(x_1)+(\phi_2(\beta))(x_2)$$ and vector field $\underline{x}$ on $T^*G$ defined by
$$\underline{x}_{\beta}\coloneqq\frac{d}{dt}\bigg\vert_{t=0}(\exp(-tx_1),\exp(-tx_2))\cdot\beta,\quad\beta\in T^*G.$$ The moment map $\phi:T^*G\longrightarrow\g^*\times\g^*$ is then completely determined by the property that \begin{equation}\label{Equation: Moment map condition}\phi^x=\theta(\underline{x})\end{equation} for all $x\in\g\times\g$.

\subsection{Definition of the system for $G=\operatorname{U}(n)$}\label{Subsection: Double GC}
We take $G$ to be $\operatorname{U}(n)$ for the balance of this paper, so that $\g=\mathfrak{u}(n)$, $\ell=n$, $\text{b}=\frac{n(n+1)}{2}$, etc. The definitions and conventions in Section \ref{Subsection: GC theory} should therefore be understood as specialized to their counterparts in Section \ref{Subsection: Special case}. One reason for continuing to use Lie type-independent notation is that it makes for a more concise and less coordinatized exposition. A second reason is that most of our results hold for compact connected Lie groups other than $\operatorname{U}(n)$. Our hope is all results in this paper can be generalized to the case of an arbitrary compact connected Lie group, e.g. via \cite{Lane}. 

Recall the map $\lambda:\g^*\longrightarrow\mathbb{R}^{\text{b}}$, subset $\g^*_{\text{s-reg}}\subset\g^*$, and rank-$\text{b}$ torus $\mathbb{T}$ discussed in the previous section. Consider the product map $(\lambda,\lambda):\g^*\times\g^*\longrightarrow\mathbb{R}^{\text{b}}\times\mathbb{R}^{\text{b}}$ and composition $$\varphi\coloneqq(\lambda,\lambda)\circ\phi=(\lambda\circ\phi_1,\lambda\circ\phi_2):T^*G\longrightarrow\mathbb{R}^{\text{b}}\times\mathbb{R}^{\text{b}}.$$ Proposition \ref{Proposition: Summary}(i) implies that $\varphi$ restricts to a moment map for a Hamiltonian action of $\mathbb{T}\times\mathbb{T}$ on the open subset
$$(T^*G)_{\text{s-reg}}\coloneqq\phi^{-1}(\g^*_{\text{s-reg}}\times \g^*_{\text{s-reg}})=\phi_1^{-1}(\g^*_{\text{s-reg}})\cap \phi_2^{-1}(\g^*_{\text{s-reg}})\subset T^*G.$$ At the same time, let us consider the subtori $$\mathbb{T}_{j1}\coloneqq \mathbb{T}_j\times\{e\}\subset\mathbb{T}\times\mathbb{T}\quad\text{and}\quad\mathbb{T}_{j2}\coloneqq\{e\}\times\mathbb{T}_j\subset\mathbb{T}\times\mathbb{T}$$ for each $j\in\{0,\ldots,m\}$. The actions of these subtori on $(T^*G)_{\text{s-reg}}$ are then given by the following consequence of Proposition \ref{Proposition: Summary}(ii): \begin{equation}\label{Equation: Action 2} (t,e)\cdot\beta=((\tau_j)_{\sigma_j(\xi)}(t),e)\cdot \beta\quad\text{and}\quad (e,t)\cdot\beta=(e,(\tau_j)_{\sigma_j(\eta)}(t))\cdot\beta\quad\end{equation} for all $j\in\{0,\ldots,m\}$, $t\in\mathbb{T}_j$, and $\beta\in (T^*G)_{\text{s-reg}}$, where $(\xi,\eta)=\phi(\beta)$.

Let us now consider the subtorus
$$\mathbb{S}\coloneqq\left\{((t_1,\ldots,t_{\ell},\underbrace{1,\ldots,1}_{\text{b}-\ell\text{ times}}),(t_{\ell},\ldots,t_1,\underbrace{1,\ldots,1}_{\text{b}-\ell\text{ times}})):(t_1,\ldots,t_{\ell})\in\mathbb{T}_0\right\}\subset\mathbb{T}\times\mathbb{T}$$
and its Lie algebra
\begin{equation}\label{Equation: Explicit Lie algebra}\mathfrak{S}\coloneqq\{((x_1,\ldots,x_{\ell},\underbrace{0,\ldots,0}_{\text{b}-\ell\text{ times}}),(x_{\ell},\ldots,x_1,\underbrace{0,\ldots,0}_{\text{b}-\ell\text{ times}})):(x_1,\ldots,x_{\ell})\in\mathbb{R}^{\ell}\}\subset\mathbb{R}^{\text{b}}\times\mathbb{R}^{\text{b}}.\end{equation}
The quotient torus
$$\mathcal{T}\coloneqq(\mathbb{T}\times\mathbb{T})/\mathbb{S}$$ then has $$\mathcal{B}\coloneqq\mathfrak{S}^{\perp}\subset\mathbb{R}^{\text{b}}\times\mathbb{R}^{\text{b}}$$ as the dual of its Lie algebra, where the orthogonal complement is with respect to the dot product. 

\begin{proposition}\label{Proposition: Trivial}
The following statements hold.
\begin{itemize}
\item[\textup{(i)}] We have $\varphi(T^*G)\subset\mathcal{B}$.
\item[\textup{(ii)}] The action of $\mathbb{T}\times\mathbb{T}$ on $(T^*G)_{\emph{s-reg}}$ descends to a free Hamiltonian action of $\mathcal{T}$ on $(T^*G)_{\emph{s-reg}}$.
\item[\textup{(iii)}] The restriction of $\varphi$ to a map $(T^*G)_{\emph{s-reg}}\longrightarrow\mathcal{B}$ is a moment map for the Hamiltonian action of $\mathcal{T}$ on $(T^*G)_{\emph{s-reg}}$.
\item[\textup{(iv)}] The non-empty fibers of this moment map are precisely the orbits of $\mathcal{T}$ in $(T^*G)_{\emph{s-reg}}$.
\end{itemize}
\end{proposition}

\begin{proof}
Composing $\varphi\big\vert_{(T^*G)_{\text{s-reg}}}$ with the orthogonal projection $\mathbb{R}^{\text{b}}\times\mathbb{R}^{\text{b}}\longrightarrow\mathfrak{S}$ yields a moment map for the Hamiltonian action of $\mathbb{S}$ on $(T^*G)_{\text{s-reg}}$. This implies that $\mathbb{S}$ acts trivially on $(T^*G)_{\text{s-reg}}$ if and only if $\varphi((T^*G)_{\text{s-reg}})\subset\mathcal{B}$. It will therefore suffice to prove (i), (iv), and the assertion that $\mathcal{T}$ acts freely on $(T^*G)_{\text{s-reg}}$.

Suppose that $(g,\xi)\in G\times\g^*=T^*G$. One has $$\lambda_k(\phi_1(g,\xi))=\lambda_k(\mathrm{Ad}_g^*(\xi))=\lambda_k(\xi)=-\lambda_{\ell+1-k}(-\xi)=-\lambda_{\ell+1-k}(\phi_2(g,\xi))$$ for all $k\in\{1,\ldots,\ell\}$. It follows that $\varphi(g,\xi)\in\mathbb{R}^{\text{b}}\times\mathbb{R}^{\text{b}}$ takes the form
\begin{equation}\label{Equation: Perp}\varphi(g,\xi)=((y_1,\ldots,y_{\ell},z_{1},\ldots,z_{\text{b}-\ell}),(-y_{\ell},\ldots,-y_1,w_1,\ldots,w_{\text{b}-\ell})),\end{equation} where $z_1,\ldots,z_{\text{b}-\ell},w_1,\ldots,w_{\text{b}-\ell}\in\mathbb{R}$ and $y_k=\lambda_k(\phi_1(g,\xi))$ for all $k\in\{1,\ldots,\ell\}$. This combines with \eqref{Equation: Explicit Lie algebra} to imply that $\varphi(T^*G)\subset\mathcal{B}$, proving (i).

We now prove that $\mathcal{T}$ acts freely on $(T^*G)_{\text{s-reg}}$. A first observation in this direction is that
\begin{align*}\mathbb{T}_{\text{comp}} & \coloneqq\left\{((t_1,\ldots,t_{\text{b}}),(\underbrace{1,\ldots,1}_{\ell\text{ times}},s_1,\ldots,s_{\text{b}-\ell})):(t_1,\ldots,t_{\text{b}})\in\mathbb{T},\text{ }(s_1,\ldots,s_{\text{b}-\ell})\in\prod_{j=1}^m\mathbb{T}_{j}\right\} \\ &=\mathbb{T}_{01}\times(\mathbb{T}_{11}\times\cdots\times \mathbb{T}_{m1})\times(\mathbb{T}_{12}\times\cdots\times \mathbb{T}_{m2})\end{align*} is a complementary subtorus to $\mathbb{S}$ in $\mathbb{T}\times\mathbb{T}$. It follows that $\mathcal{T}$ acts freely on $(T^*G)_{\text{s-reg}}$ if and only if $\mathbb{T}_{\text{comp}}$ acts freely on $(T^*G)_{\text{s-reg}}$. 

Suppose that $\beta\in (T^*G)_{\text{s-reg}}$. Since $\phi\big\vert_{(T^*G)_{\text{s-reg}}}:(T^*G)_{\text{s-reg}}\longrightarrow\g^*_{\text{s-reg}}\times\g^*_{\text{s-reg}}$ is $\mathbb{T}\times\mathbb{T}$-equivariant and  $\mathbb{T}_{1}\times\cdots\times\mathbb{T}_{m}\subset\mathbb{T}$ acts freely on $\g^*_{\text{s-reg}}$, the $\mathbb{T}_{\text{comp}}$-stabilizer of $\beta$ is necessarily contained in $\mathbb{T}_{01}$. Equation \eqref{Equation: Action 2} allows one to identify the $\mathbb{T}_{01}$-stabilizer of $\beta$ with the $G_{\xi}$-stabilizer of $\beta$, where $\xi=\phi_1(\beta)\in\g^*_{\text{s-reg}}$. The latter stabilizer is trivial, as $G\times\{e\}$ acts freely on $T^*G$. These last few sentences collectively imply that the $\mathbb{T}_{\text{comp}}$-stabilizer of $\beta$ is trivial. In light of the previous paragraph, we conclude that $\mathcal{T}$ acts freely on $(T^*G)_{\text{s-reg}}$.

It remains only to prove (iv). This would follow from the symplectic quotient $(T^*G)_{\text{s-reg}}\sll{x}\mathcal{T}$ being a point for all $x\in\varphi((T^*G)_{\text{s-reg}})$. An equivalent condition would be for $(T^*G)_{\text{s-reg}}\sll{x}\mathbb{T}_{\text{comp}}$ to be a point for all $x\in\varphi_{\text{comp}}((T^*G)_{\text{s-reg}})$, where 
$$\varphi_{\text{comp}}:T^*G\longrightarrow\mathbb{R}^{\text{b}}\times\mathbb{R}^{\text{b}-\ell}$$ is obtained by composing $\varphi:T^*G\longrightarrow\mathbb{R}^{\text{b}}\times\mathbb{R}^{\text{b}}$ with the orthogonal projection onto the subspace $\mathbb{R}^{\text{b}}\times(\{0\}\times\mathbb{R}^{\text{b}-\ell})=\mathbb{R}^{\text{b}}\times\mathbb{R}^{\text{b}-\ell}$. To this end, we fix $x=(y_1,\ldots,y_{\text{b}},z_1,\ldots,z_{\text{b}-\ell})\in\varphi_{\text{comp}}((T^*G)_{\text{s-reg}})$. Consider the element $y\coloneqq(y_1,\ldots,y_{\ell})\in\mathbb{R}^{\ell}$, diagonal matrix $\xi\coloneqq\mathrm{diag}(y_1,\ldots,y_{\ell})\in\g^*_{\text{s-reg}}$, and coadjoint orbits $\mathcal{O}_\xi\coloneqq G\cdot\xi\subset\g^*$ and $\mathcal{O}_{-\xi}\coloneqq G\cdot (-\xi)\subset\g^*$. The moment map $\phi:T^*G\longrightarrow\g^*\times\g^*$ then restricts and descends to a symplectomorphism
$$(T^*G)_{\text{s-reg}}\sll{y}\mathbb{T}_{01}\longrightarrow(\mathcal{O}_{\xi}\cap\g^*_{\text{s-reg}})\times (\mathcal{O}_{-\xi}\cap\g^*_{\text{s-reg}}),\quad[(g,\eta)]\mapsto (\mathrm{Ad}_g^*(\eta),-\eta).$$ Since $\phi\big\vert_{(T^*G)_{\text{s-reg}}}$ is $(\mathbb{T}\times\mathbb{T})$-equivariant, our symplectomorphism is equivariant for the action of the subtorus $$(\mathbb{T}_{11}\times\cdots\times \mathbb{T}_{m1})\times(\mathbb{T}_{12}\times\cdots\times \mathbb{T}_{m2})\subset\mathbb{T}\times\mathbb{T}.$$ We also know each non-empty symplectic quotient of $(\mathcal{O}_{\xi}\cap\g^*_{\text{s-reg}})\times (\mathcal{O}_{-\xi}\cap\g^*_{\text{s-reg}})$ by this subtorus to be a point \cite[Section 5]{GuilleminSternbergGC}. It follows that any non-empty symplectic quotient of $(T^*G)_{\text{s-reg}}\sll{y}\mathbb{T}_{01}$ by the same subtorus is also a point. In other words, the symplectic quotient $(T^*G)_{\text{s-reg}}\sll{x}\mathbb{T}_{\text{comp}}$ is a point.
\end{proof}

We will henceforth regard $\varphi$ as a map $T^*G\longrightarrow\mathcal{B}$. Proposition \ref{Proposition: Trivial} implies that the restriction of this map to $(T^*G)_{\text{s-reg}}$ is a completely integrable system. The following terminology then captures the relation of $\varphi=(\lambda,\lambda)\circ\phi$ to the Gelfand--Cetlin system $\lambda:\g^*\longrightarrow\mathbb{R}^{\text{b}}$.    

\begin{definition}
We call $\varphi:T^*G\longrightarrow\mathcal{B}$ the \textit{double Gelfand--Cetlin system} on $T^*G$.
\end{definition}

\subsection{Moment map images}\label{Subsection: Images}
We now establish two facts concerning images under the double Gelfand--Cetlin system $\varphi:T^*G\longrightarrow\mathcal{B}\subset\mathbb{R}^{2\text{b}}$. 

\begin{proposition}\label{Proposition: Open dense image}
The image $\varphi((T^*G)_{\emph{s-reg}})\subset\mathbb{R}^{2\emph{b}}$ consists of all $(\alpha_1,\ldots,\alpha_{2\emph{b}})\in\mathbb{R}^{2\emph{b}}$ that satisfy the following conditions:
\begin{itemize}
\item[\textup{(i)}] $\alpha_1>\cdots>\alpha_{n}$;
\item[\textup{(ii)}] $\alpha_{\emph{b}+1+j}=-\alpha_{\ell-j}$ for all $j\in\{0,\ldots,\ell-1\}$;
\item[\textup{(iii)}] $(\alpha_{\ell+1},\ldots,\alpha_{\emph{b}})\in\mathrm{interior}(\mathrm{GC}_{\alpha})$, where $\alpha\coloneqq(\alpha_1,\ldots,\alpha_{\ell})$;
\item[\textup{(iv)}] $(\alpha_{\emph{b}+\ell+1},\ldots,\alpha_{\emph{2b}})\in\mathrm{interior}(\mathrm{GC}_{\alpha^*})$, where $\alpha^*\coloneqq(-\alpha_{\ell},\ldots,-\alpha_1)$.
\end{itemize}
\end{proposition}

\begin{proof}
Suppose that $(g,\xi)\in (G\times\g^*)_{\text{s-reg}}=(T^*G)_{\text{s-reg}}$, and consider the element $\varphi(g,\xi)=(\alpha_1,\ldots,\alpha_{2\text{b}})$. By virtue of \eqref{Equation: Perp}, (ii) is satisfied. The description \eqref{Equation: SR} then implies that (i), (iii), and (iv) hold. It therefore remains only to prove that any element of $\mathbb{R}^{2\text{b}}$ satisfying (i)-(iv) is in $\varphi((T^*G)_{\text{s-reg}})$. 

Let $(\alpha_1,\ldots,\alpha_{2\text{b}})\in\mathbb{R}^{2\text{b}}$ be such that (i)-(iv) hold. Since $$\alpha^*=(-\alpha_{\ell},\ldots,-\alpha_1)=(\alpha_{\text{b}+1},\cdots,\alpha_{\text{b}+\ell})$$ is a strictly decreasing sequence and $(\alpha_{\text{b}+\ell+1},\ldots,\alpha_{\text{2b}})\in\mathrm{interior}(\mathrm{GC}_{\alpha^*})$, there exists $\xi\in\mathcal{O}_{\alpha^*}\cap\g^*_{\text{s-reg}}$ satisfying $\lambda(\xi)=(\alpha_{\text{b}+1},\ldots,\alpha_{2\text{b}})$. It follows that $$-\xi\in(-\mathcal{O}_{\alpha^*})\cap\g^*_{\text{s-reg}}=\mathcal{O}_{\alpha}\cap\g_{\text{s-reg}}^*.$$ On the other hand, (i) and (iii) imply the existence of $\eta\in\mathcal{O}_{\alpha}\cap\g^*_{\text{s-reg}}$ satisfying $\lambda(\eta)=(\alpha_1,\ldots,\alpha_{\text{b}})$. We may therefore find $g\in G$ such that $\mathrm{Ad}_g(-\xi)=\eta$. The point $(g,-\xi)\in G\times\g^*=T^*G$ then belongs to $(T^*G)_{\text{s-reg}}$ and satisfies $\varphi(g,-\xi)=(\alpha_1,\ldots,\alpha_{2\text{b}})$. This completes the proof.
\end{proof}

An analogous and slightly simpler argument yields the image of the double Gelfand--Cetlin system.

\begin{proposition}\label{Proposition: Full image}
The image $\varphi(T^*G)\subset\mathbb{R}^{2\emph{b}}$ consists of all $(\alpha_1,\ldots,\alpha_{2\emph{b}})\in\mathbb{R}^{2\emph{b}}$ that satisfy the following conditions:
\begin{itemize}
\item[\textup{(i)}] $\alpha_1\geq\cdots\geq\alpha_{\ell}$;
\item[\textup{(ii)}] $\alpha_{\emph{b}+1+j}=-\alpha_{\ell-j}$ for all $j\in\{0,\ldots,\ell-1\}$;
\item[\textup{(iii)}] $(\alpha_{\ell+1},\ldots,\alpha_{\emph{b}})\in\mathrm{GC}_{\alpha}$, where $\alpha\coloneqq(\alpha_1,\ldots,\alpha_{\ell})$;
\item[\textup{(iv)}] $(\alpha_{\emph{b}+\ell+1},\ldots,\alpha_{\emph{2b}})\in\mathrm{GC}_{\alpha^*}$, where $\alpha^*\coloneqq(-\alpha_{\ell},\ldots,-\alpha_1)$.
\end{itemize}
\end{proposition}

\section{A geometric quantization of $T^*\hspace{-2pt}\operatorname{U}(n)$}\label{Section: Quantization}

\subsection{The Bohr--Sommerfeld set}\label{Subsection: BS} Consider the trivial Hermitian line bundle $\mathcal{L}\coloneqq T^*G\times\mathbb{C}\longrightarrow T^*G$ and section $$s:T^*G\longrightarrow\mathcal{L},\quad\beta\mapsto(\beta,1).$$ Let $\nabla$ be the connection on $\mathcal{L}$ satisfying $$\nabla_X(fs)=(X(f)-i\theta(X)f)s$$ for all smooth complex vector fields $X$ on $T^*G$ and smooth functions $f:T^*G\longrightarrow\mathbb{C}$, where $\theta$ is regarded as a complex one-form in the obvious way. As $\omega$ is the curvature of $\nabla$, $(\mathcal{L},\nabla)$ is a prequantum line bundle on $T^*G$.

Proposition \ref{Proposition: Trivial} implies that the restricted double Gelfand--Cetlin system $\varphi\big\vert_{(T^*G)_{\text{s-reg}}}:(T^*G)_{\text{s-reg}}\longrightarrow\mathcal{B}$ constitute a real polarization  of $(T^*G)_{\text{s-reg}}$. In this context, we have the following result.

\begin{theorem}\label{Theorem: BS}
A point $\beta\in (T^*G)_{\emph{s-reg}}$ satisfies $\varphi(\beta)\in\mathbb{Z}^{2\emph{b}}$ if and only if $\varphi(\beta)$ is a Bohr--Sommerfeld point of the restricted Gelfand--Cetlin system.
\end{theorem}

\begin{proof}
Observe that the Lie algebra of $\mathbb{T}_{\text{comp}}$ is 
$$\left\{((x_1,\ldots,x_{\text{b}}),(\underbrace{0,\ldots,0}_{\ell\text{ times}},y_1,\ldots,y_{\text{b}-\ell})):(x_1,\ldots,x_{\text{b}})\in\mathbb{R}^{\text{b}},\text{ }(y_1,\ldots,y_{\text{b}-\ell})\in\mathbb{R}^{\text{b}-\ell}\right\}\cong\mathbb{R}^{\text{b}}\times\mathbb{R}^{\text{b}-\ell},$$ where $\mathbb{T}_{\text{comp}}$ is defined in the proof of Proposition \ref{Proposition: Trivial}. Let us also note that $$\varphi=(\lambda_1\circ\phi_1,\ldots,\lambda_{\text{b}}\circ\phi_1,\lambda_1\circ\phi_2,\ldots,\lambda_{\text{b}}\circ\phi_2):T^*G\longrightarrow\mathbb{R}^{\text{b}}\times\mathbb{R}^{\text{b}}$$ is the decomposition of $\varphi$ into component functions. It follows that 
\begin{equation}\label{Equation: Comp}\varphi_{\text{comp}}\coloneqq(\lambda_1\circ\phi_1,\ldots,\lambda_{\text{b}}\circ\phi_1,\lambda_{\ell+1}\circ\phi_2,\ldots,\lambda_{\text{b}}\circ\phi_2):T^*G\longrightarrow\mathbb{R}^{\text{b}}\times\mathbb{R}^{\text{b}-\ell}\end{equation} is a moment map for the Hamiltonian action of $\mathbb{T}_{\text{comp}}\subset\mathbb{T}\times\mathbb{T}$. This proof also implies that $$\lambda_k(\phi_1(\beta))=-\lambda_{\ell+1-k}(\phi_2(\beta))$$ for all $\beta\in T^*G$ and $k\in\{1,\ldots,\ell\}$. We conclude that $\varphi(\beta)\in\mathbb{Z}^{2\text{b}}$ if and only if $\varphi_{\mathrm{comp}}(\beta)\in\mathbb{Z}^{\text{b}}\times\mathbb{Z}^{\text{b}-\ell}$. It will therefore suffice to prove that $\beta\in (T^*G)_{\text{s-reg}}$ satisfies $\varphi_{\text{comp}}(\beta)\in\mathbb{Z}^{\text{b}}\times\mathbb{Z}^{\text{b}-\ell}$ if and only if $\varphi(\beta)$ is a Bohr--Sommerfeld point. In light of \cite[Theorem 2.4]{GuilleminSternbergGC}, we are further reduced to establishing the existence of $\beta\in (T^*G)_{\text{s-reg}}$ for which $\varphi_{\mathrm{comp}}(\beta)\in\mathbb{Z}^{\text{b}}\times\mathbb{Z}^{\text{b}-\ell}$ and $\varphi(\beta)$ is a Bohr--Sommerfeld point.

Suppose that $\beta\in (T^*G)_{\text{s-reg}}$ satisfies $\varphi_{\mathrm{comp}}(\beta)\in\mathbb{Z}^{\text{b}}\times\mathbb{Z}^{\text{b}-\ell}$, and recall the decomposition
\begin{equation}\label{Equation: Decomposition}\mathbb{T}_{\text{comp}}=\mathbb{T}_{01}\times\cdots\times\mathbb{T}_{m1}\times\mathbb{T}_{12}\times\cdots\times\mathbb{T}_{m2}\end{equation} of $\mathbb{T}_{\text{comp}}$ into components. Consider a specific component, i.e. $\mathbb{T}_{jk}$ for $j\in\{0,\ldots,m\}$ and $k\in\{1,2\}$ with $(j,k)\neq (0,2)$. The orbit $\Lambda_{jk}\coloneqq\mathbb{T}_{jk}\cdot\beta\subset(T^*G)_{\text{s-reg}}$ is an isotropic submanifold that we now examine in detail.

Let $\phi(\beta)=(\xi,\eta)\in\g^*_{\text{s-reg}}\times\g^*_{\text{s-reg}}$, and consider the Hamiltonian actions of the subgroups $$(G_j)_{\sigma_j(\xi)}\cong(G_j)_{\sigma_j(\xi)}\times\{e\}\subset G\times G\quad\text{and}\quad (G_j)_{\sigma_j(\eta)}\cong\{e\}\times (G_j)_{\sigma_j(\eta)}\subset G\times G$$ on $T^*G$. Equation \eqref{Equation: Action 2} implies that
$$\Lambda_{jk}=(G_j)_{\sigma_j(\xi)}\cdot\beta\hspace{5pt}\text{if }k=1\quad\text{and}\quad\Lambda_{jk}=(G_j)_{\sigma_j(\eta)}\cdot\beta\hspace{5pt}\text{if }k=2.$$ One further observation is that the actions of $(G_j)_{\sigma_j(\xi)}$ and $(G_j)_{\sigma_j(\eta)}$ admit respective moment maps of
$$\mu_1:(T^*G)_{\text{s-reg}}\longrightarrow(\g_j)_{\sigma_j(\xi)}^*,\quad\beta\mapsto\phi_1(\beta)\bigg\vert_{(\g_j)_{\sigma_j(\xi)}}$$ and  $$\mu_2:(T^*G)_{\text{s-reg}}\longrightarrow(\g_j)_{\sigma_j(\eta)}^*,\quad\beta\mapsto\phi_2(\beta)\bigg\vert_{(\g_j)_{\sigma_j(\eta)}}.$$ It follows that $$\mu_1(\Lambda_{jk})=\left\{\xi\big\vert_{(\g_j)_{\sigma_j(\xi)}}\right\}\hspace{5pt}\text{if }k=1\quad\text{and}\quad \mu_2(\Lambda_{jk})=\left\{\eta\big\vert_{(\g_j)_{\sigma_j(\eta)}}\right\}\hspace{5pt}\text{if }k=2.$$ On the other hand, Proposition \ref{Proposition: Integral} and Equation \eqref{Equation: Comp} combine with the fact that $\varphi_{\text{comp}}(\beta)\in\mathbb{Z}^{\text{b}}\times\mathbb{Z}^{\text{b}-\ell}$ to imply that $\xi\big\vert_{(\g_j)_{\sigma_j(\xi)}}$ and $\eta\big\vert_{(\g_j)_{\sigma_j(\eta)}}$ belong to the weight lattices of $(G_j)_{\sigma_j(\xi)}$ and $(G_j)_{\sigma_j(\eta)}$, respectively. We also recognize that the actions of $(G_j)_{\sigma_j(\xi)}$ and $(G_j)_{\sigma_j(\eta)}$ preserve the tautological one-form on $T^*G$, and so lift to connection-preserving actions on the prequantum line bundle. It is also clear from \eqref{Equation: Moment map condition} that $$\mu_1^x=\theta(\underline{x})\quad\text{and}\quad\mu_2^y=\theta(\underline{y})$$ for all $x\in(\g_j)_{\sigma_j(\xi)}$ and $y\in (\g_j)_{\sigma_j(\eta)}$. These last four sentences imply that $\Lambda_{jk}$ is integral with respect to the prequantum line bundle, as in \cite[Section 12.4]{guillemin-sternberg-semi-classical} and \cite[Section 2.5]{GuilleminUribeWang}. This amounts to the existence of a smooth map $f_{jk}:\Lambda_{jk}\longrightarrow\operatorname{U}(1)$ satisfying \begin{equation}\label{Equation: BS isotropic} \iota_{\Lambda_{jk}}^*\theta=\frac{1}{2\pi i}f_{jk}^{-1}\mathrm{d}f_{jk},\end{equation} where $\iota_{\Lambda_{jk}}:\Lambda_{jk}\longrightarrow T^*G$ is the inclusion map. 

Let us consider the Lagrangian submanifold $$\Lambda\coloneqq\mathbb{T}_{\text{comp}}\cdot\beta\subset (T^*G)_{\text{s-reg}}.$$ The decomposition \eqref{Equation: Decomposition} then gives rise to a diffeomorphism
$$\Lambda\cong\prod_{(j,k)\neq (0,2)}(\mathbb{T}_{jk}\cdot\beta)=\prod_{(j,k)\neq (0,2)}\Lambda_{jk},$$ and hence also to projections $\pi_{jk}:\Lambda\longrightarrow\Lambda_{jk}$. This allows us to consider the pullbacks $\pi_{jk}^*f_{jk}:\Lambda\longrightarrow\operatorname{U}(1)$ and their pointwise product $$f\coloneqq\left(\prod_{(j,k)\neq (0,2)}\pi_{jk}^*f_{jk}\right):\Lambda\longrightarrow\operatorname{U}(1).$$ The identity \eqref{Equation: BS isotropic} and a straightforward calculation then yield $$\iota_{\Lambda}^*\theta=\frac{1}{2\pi i}f^{-1}\mathrm{d}f,$$ where $\iota_{\Lambda}:\Lambda\longrightarrow T^*G$ is the inclusion map. In particular, $\varphi(\beta)$ is a Bohr--Sommerfeld point.

The preceding paragraphs establish that any $\beta\in (T^*G)_{\text{s-reg}}$ satisfying $\varphi_{\mathrm{comp}}(\beta)\in\mathbb{Z}^{\text{b}}\times\mathbb{Z}^{\text{b}-\ell}$ is such that $\varphi(\beta)$ is a Bohr--Sommerfeld point. In light of the first paragraph of this proof, it remains only to verify the existence of $\beta\in (T^*G)_{\text{s-reg}}$ with $\varphi_{\mathrm{comp}}(\beta)\in\mathbb{Z}^{\text{b}}\times\mathbb{Z}^{\text{b}-\ell}$. We begin by choosing an $\ell\times\ell=n\times n$ diagonal matrix $\xi$ with pairwise distinct integral eigenvalues. It follows that $\xi,-\xi\in\g^*_{\text{s-reg}}$, implying that $(e,\xi)\in (G\times\g^*)_{\text{s-reg}}=(T^*G)_{\text{s-reg}}$. On the other hand, it is clear that $\varphi_{\mathrm{comp}}(e,\xi)\in\mathbb{Z}^{\text{b}}\times\mathbb{Z}^{\text{b}-\ell}$.
\end{proof}

\begin{corollary}\label{Corollary: BS}
A point $(\alpha_1,\ldots,\alpha_{2\emph{b}})\in\mathbb{R}^{2\emph{b}}$ belongs to the Bohr--Sommerfeld set of the restricted double Gelfand--Cetlin system $\varphi:(T^*G)_{\emph{s-reg}}\longrightarrow\mathcal{B}\subset\mathbb{R}^{2\emph{b}}$ if and only if it satisfies the following properties:
\begin{itemize}
\item[\textup{(i)}] $(\alpha_1,\ldots,\alpha_{2\emph{b}})\in\mathbb{Z}^{2\emph{b}}$;
\item[\textup{(ii)}] $\alpha_1>\cdots>\alpha_{\ell}$;
\item[\textup{(iii)}] $\alpha_{\emph{b}+1+j}=-\alpha_{\ell-j}$ for all $j\in\{0,\ldots,\ell-1\}$;
\item[\textup{(iv)}] $(\alpha_{\ell+1},\ldots,\alpha_{\emph{b}})\in\mathrm{interior}(\mathrm{GC}_{\alpha})$, where $\alpha\coloneqq(\alpha_1,\ldots,\alpha_{\ell})$;
\item[\textup{(v)}] $(\alpha_{\emph{b}+\ell+1},\ldots,\alpha_{\emph{2b}})\in\mathrm{interior}(\mathrm{GC}_{\alpha^*})$, where $\alpha^*\coloneqq(-\alpha_{\ell},\ldots,-\alpha_1)$.
\end{itemize}
\end{corollary}

\begin{proof}
This follows immediately from Proposition \ref{Proposition: Open dense image} and Theorem \ref{Theorem: BS}.
\end{proof}

\subsection{The Peter--Weyl theorem and invariance of polarization}\label{Subsection: Peter--Weyl}
As explained in Section \ref{Subsection: The quantization}, we will quantize $T^*G$ by replacing the Bohr--Sommerfeld set $\varphi((T^*G)_{\text{s-reg}})\cap\mathbb{Z}^{2\text{b}}$ with $\varphi(T^*G)\cap\mathbb{Z}^{2\text{b}}$. Proposition \ref{Proposition: Full image} tells us that this new Bohr--Sommerfeld set is the locus of $(\alpha_1,\ldots,\alpha_{2\text{b}})\in\mathbb{R}^{2\text{b}}$ satisfying conditions (i) and (iii) in Corollary \ref{Corollary: BS}, as well as $\alpha_1\geq\cdots\geq\alpha_{\ell}$, $(\alpha_{\ell+1},\ldots,\alpha_{\text{b}})\in\mathrm{GC}_{\alpha}$, and $(\alpha_{\text{b}+\ell+1},\ldots,\alpha_{\text{2b}})\in\mathrm{GC}_{\alpha^*}$. A Bohr--Sommerfeld point is thereby the data of a dominant integral weight $\alpha$ of $G$, an integral element of $\mathrm{GC}_{\alpha}$, and an integral element of $\mathrm{GC}_{\alpha^*}$. At the same time, let us write $V_{\alpha}$ and $V_{\alpha^*}$ for the irreducible complex $G$-modules of highest weights $\alpha$ and $\alpha^*$, respectively. An application of \cite[Theorem 6.1]{GuilleminSternbergGC} reveals that the number of integral points in $\mathrm{GC}_{\alpha}$ (resp. $\mathrm{GC}_{\alpha^*}$) is precisely the dimension of $V_{\alpha}$ (resp. $V_{\alpha^*}$). We also note that $V_{\alpha^*}=V_{\alpha}^*$. The double Gelfand--Cetlin system $\varphi:T^*G\longrightarrow\mathcal{B}$ thereby gives rise to the quantization $$Q(T^*G)=\bigoplus_{\alpha}V_{\alpha}\otimes V_{\alpha}^*,$$ where $\alpha$ runs over the dominant integral weights of $G$. An application of the Peter--Weyl theorem allows us to identify this quantization with $L^2(G)$, the geometric quantization of $T^*G$ arising from the bundle projection $T^*G\longrightarrow G$. Therefore, the Peter--Weyl theorem may be interpreted in terms of geometric quantization as giving an instance of invariance of polarization. 

\bibliographystyle{acm} 
\bibliography{Quantization}

\end{document}